\documentclass[12pt]{article}
\usepackage{amsmath, amsfonts, latexsym, amssymb, graphicx, mathtools,  wasysym, enumitem, pifont}

\usepackage[all]{xy}
\usepackage{hyperref}

\hypersetup{colorlinks,citecolor=blue,linkcolor=blue}

\title{The Lamplighter Group is Not Semistable at Infinity}
\author{Michael Mihalik}
\usepackage{tikz}
\newtheorem{theorem}{Theorem}[section]

\newtheorem{lemma}[theorem]{Lemma}

\newtheorem{remark}[theorem]{Remark}

\newcounter{claimnum}

\newcounter{definitionnum}

\newenvironment{proof}{\addvspace{12pt}\noindent{\bf Proof:}}{
$\Box$\par\addvspace{12pt}}

\date{\today}
\begin{document}
\maketitle

\begin{abstract}
The question of whether or not all finitely presented groups are semistable at infinity has been studied for over 40 years. In 1986,  %\cite{M4}  
we defined what it means for a finitely generated group to be semistable at infinity - in analogy with the definition for finitely presented groups. At that time
we suggest that the Lamplighter group may not be semistable at $\infty$, but until now there was no confirmed example of a finitely generated group that is not semistable at $\infty$. We prove the Lamplighter group is not semistable at $\infty$. Finitely generated non-semistable groups may be important in finding non-semistable finitely presented groups via ascending HNN extensions. There is an ascending HNN extension $E$ of the Lamplighter group (called the Extended Lamplighter group) that is finitely presented. It would seem that $E$ is a candidate to be a finitely presented non-semistable at $\infty$ group, but a result of N. Silkin %\cite{Silk1},
 shows that $E$ is in fact simply connected at $\infty$.
\end{abstract}

\section {Introduction}\label{Intro}

The question of whether or not all finitely presented groups are semistable at infinity has been studied for over 40 years and is one of the premier questions in the asymptotic study of finitely presented groups. If a finitely presented one ended group $G$ is semistable at infinity then the fundamental group at infinity of $G$ can be defined unambiguously. The semistability of a finitely presented group $G$ implies that $H^2(G,\mathbb ZG)$ is free abelian. The question of whether or not all finitely presented groups $G$ are such that $H^2(G,\mathbb ZG)$ is free abelian is attributed to H. Hopf and remains unanswered. The class of finitely presented groups that are semistable contains many classes of important groups including: Word hyperbolic groups (combining work of Bestvina-Mess 
\cite{BM91}, Bowditch \cite{Bow99B}, 
G. Levitt \cite{Lev98} and G. Swarup \cite{Swarup}), CAT(0) cube groups S. Shepherd \cite{SSh}, 
1-relator groups \cite {MT92}, 
general Coxeter and Artin groups  \cite{MM96}, 
most solvable groups \cite{M4} and various group extensions and ascending HNN extensions \cite{M7}. In 1986  \cite {M4} 
we defined semistability for finitely generated groups as an aid in showing certain finitely presented groups are semistable. While ascending HNN extensions of finitely presented groups are semistable, the class of finitely presented groups that are ascending HNN-extensions of finitely generated (but not finitely presented) groups may contain non-semistable finitely presented groups. If a finitely generated group is semistable then any finitely presented ascending HNN extension of that group will be semistable \cite{M4}. 
Hence it is important to find non-semistable finitely generated groups that are base groups for finitely presented ascending HNN-extensions (see \cite {M7} 
for a current analysis). Our main result is that the Lamplighter group $L$ is not semistable. Unfortunately, interesting ascending HNN-extensions of this group are not only semistable, but simply connected at infinity. Section \ref{BD} of the paper contains basic definitions and establishes the types of spaces we work in. We prove our main theorem in Section \ref{MT} - The Lamplighter Group is not semistable at infinity. Other potentially non-semistable finitely generated groups suggested by D. Osin are mentioned in Section \ref{O}. We describe the Extended Lamplighter Group in Section \ref{E}. This is a finitely presented ascending HNN extension of the Lamplighter group. We observe that not only is it semistable, but in fact simply connected at infinity. 

\section {Basic Definitions} \label{BD} 
All groups and spaces here are 1-ended and all spaces are locally finite CW-complexes, so we only give the definitions of interest for such groups and spaces. A proper {\it edge path ray} in a space $X$ is a proper map $r:[0,\infty)\to X$ such that for each integer $n\geq 0$, $r|_{[n,n+1]}$ is a linear homeomorphism to an edge of $X$. A connected 1-ended locally finite CW complex $X$ is {\it semistable at infinity} if any two proper edge path rays in $X$ are properly homotopic. Equivalently, given any proper edge path ray $r$ in $X$ and finite subcomplex $C\subset X$ there is a finite complex $D\subset X$ such that for any third finite complex $E\subset X$ and edge path loop $\alpha$ based on $r$ and with image in $X-D$, the loop $\alpha$ is homotopic rel $r$ to a loop in $X-E$ by a homotopy in $X-C$ (so ``far out" loops on $r$ can be pushed arbitrarily far out along $r$ by far out homotopies). A  finitely presented 1-ended group $G$ is {\it semistable at infinity} if for some (any) finite complex $Y$ with $\pi_1(Y)=G$, the universal cover of $Y$ is semistable at infinity. 

Suppose $G$ is a group with finite generating set $S$. Let $\phi:F(S)\to G$ be the quotient map sending $s$ to $s$ for all $s\in S$. For $A\subset F(S)$, let $N(A)$ be the normal closure of $A$ in $F(S)$. Let $\Gamma(G, S)$ be the Cayley graph of $G$ with respect to $S$. If $\mathcal R=\{r_1,\ldots, r_m\}\subset ker(\phi)$ is a finite set of $S$-relations for $G$, then let $\Gamma(G, S,\mathcal R)$ be obtained from $\Gamma(G,S)$ by attaching a 2-cell at each vertex of $\Gamma(G, S)$ according to the labeling of each $r\in \mathcal R$. So there are $m$ 2-cells attached to $\Gamma(G, S)$ at each vertex of $\Gamma(G, S)$. Observe that if $r\in N(\mathcal R)\subset F(S)$ then (in $F(S)$), $r$ is a product of conjugates of elements of $\mathcal R^{\pm 1}$.
This implies that any edge path loop in $\Gamma(G,S,\mathcal R)$ with edge labeling equal to $r$ is homotopically trivial in $\Gamma(G,S,\mathcal R)$. We now show the converse of this statement.

Suppose the group $G_1$ has finite presentation $\langle S:\mathcal R\rangle$. If $\phi_1:F(S)\to G_1$ is the quotient map sending $s$ to $s$ for all $s\in S$, then 
$$N(\mathcal R)=ker (\phi_1)\subset ker(\phi),$$ $\phi_1(ker(\phi))$ is a normal subgroup of $G_1$ and $G_1/\phi_1(ker \phi)=G$. Let $\Gamma_1$ be the Cayley 2-complex for $\langle S:\mathcal R\rangle $ (and hence the universal cover of the finite quotient complex $G_1\backslash \Gamma_1$). The group of covering transformations of $\Gamma_1$ is $G_1$. The space $\Gamma(G,S, \mathcal R)$ is the quotient of $\Gamma_1$ by the left action of $\phi_1(ker(\phi))$ (with $G$ as a group of covering transformations). We have covering maps:

$$\Gamma_1\buildrel {\backslash \phi_1(ker(\phi))}\over \longrightarrow \Gamma(G,S,\mathcal R)\buildrel \backslash G\over \longrightarrow G_1\backslash \Gamma_1.$$
We also have the commutative diagram:

\[
\xymatrix{
F(S) \ar[d]^{\phi_1} \ar[dr]^\phi & \\
G_1 \ar[r] & G.
 }
\]

\begin{lemma} \label{L1} 
An edge path loop $\alpha$ in $\Gamma(G,S,\mathcal R)$ (with labeling $r\in ker(\phi)$) is homotopically trivial in $\Gamma(G,S,\mathcal R)$ if and only if $r\in N(\mathcal R)(\subset F(S))$.
\end{lemma}
\begin{proof}
We have already established one part of the lemma. If $\alpha$ is homotopically trivial in $\Gamma(G,S,\mathcal R)$ then the homotopy lifting theorem implies that $\alpha$ lifts to a (homotopically trivial) loop in $\Gamma_1$ (with the same labeling). Any edge path loop in $\Gamma_1$ has edge labeling in $N(\mathcal R)$.
\end{proof}

Observe that if $r$ is in the normal closure of $\mathcal R$ (so that $r$ is a product of conjugates of elements in $\mathcal R^{\pm 1}$) then it is elementary to form a van Kampen diagram for $r$ with 2-cells labeled by elements of $\mathcal R^{\pm 1}$ (simply draw an edge path loop in the plane with this labeling as a product of conjugates of elements of $\mathcal R^{\pm 1}$).

The finitely generated group $G$ is {\it semistable at } $\infty$ if there is a finite generating $S$ for $G$ and a finite set $\mathcal R\subset ker(\phi)$ of $S$-relations for $G$ such that $\Gamma(G, S,\mathcal R)$ is semistable at $\infty$. This definition is independent of finite generating set. If $\mathcal R\subset \mathcal R'\subset ker(\phi)$ and $\Gamma(G, S, \mathcal R)$ is semistable at $\infty$, then $\Gamma(G, S, \mathcal R')$ is semistable at $\infty$ as well (see \cite{M4}). 

\section{Non-Semistability of the Lamplighter Group}\label{MT}

The {\it Lamplighter Group} has presentation 
$$L=\langle a,x: a^2=1, [a, x^{-k}ax^k]=1 \hbox{ for all integers } k\rangle.$$

Since $a$ has order 2, $x^{-k}ax^k$ has order 2 as well. Hence the commutation relation $[a, x^{-k}ax^k]$ can be written as $a(x^{-k}ax^k)a(x^{-k}ax^k)$. Conjugating this relation by $x^{-k}$ gives the relation $[a, x^{k}ax^{-k}]$ so that for our purposes we need not include both $[a, x^{-k}ax^k]$ and $[a, x^{k}ax^{-k}]$ in a set of relations $\mathcal R$. 

\vspace {.5in}
\vbox to 2in{\vspace {-2in} \hspace {-.3in}
\hspace{-.5 in}
\includegraphics[scale=1]{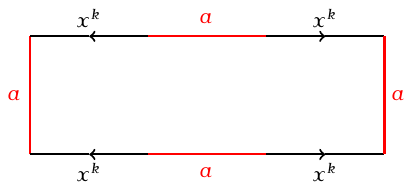}
\vss }

\vspace{-1in}

\centerline{Figure 1}

Any finite set of relations is a subset of those determined by $\mathcal R_k=\{a^2, [a, x^{-1}ax], \ldots, [a, x^{-k}ax^k]\}$ for some integer $k>0$. 
Hence we need only show that $\Gamma_k=\Gamma(L, \{a, x\}, \mathcal R_k)$ is not semistable at $\infty$ for all $k>0$. 
The boundary of each cell of $\Gamma_k$ arising from a commutation relation, has exactly four edges labeled $a$. Each $a$-edge is diametrically opposite to an edge labeled $a$ (see Figure 1).

Suppose an edge path loop $\alpha$ in $\Gamma_k$ is homotopically trivial in $\Gamma_k$. Then there is a van Kampen diagram $\mathcal D$ for $\alpha$ whose cells are bounded by edge path loops with labels in $\mathcal R_{k}$. Define an $a$-{\it band} in $\mathcal D$ as follows. If $e_1$ is an edge in the boundary of $\mathcal D$ with label $a$  then there is exactly one 2-cell $C_1$, containing $e_1$. If $e_2$ is  opposite $e_1$ in $C_1$ then the label of $e_2$ is also $a$. If $e_2$ is not on the boundary of $\mathcal D$, there is a unique cell $C_2\ne C_1$ containing $e_2$ with opposite edge labeled $a$. String cells together in this way until $e_i$ is also an edge of the boundary of $\mathcal D$. This collection of cells is a {\it band}. The edges $e_i$ are called {\it connecting} edges of the band. 

If $e_1$ is in the interior of $\mathcal D$, then there are exactly two  2-cells containing $e_1$. Suppose $e_1$ is not a connecting edge of any band where the band begins and ends on the boundary of $\mathcal D$. Say $C_1$ is one of the two cells containing $e_1$. Let $e_2$ be the edge of $C_1$ opposite $e_1$ as above, let $C_2\ne C_1$ be the other cell of $\mathcal D$ containing $e_2$. Stringing cells together, eventually $e_i=e_1$ for some $i>1$. This loop of cells is also called a band and the $e_i$ are connecting edges of the band. 
Every $a$-edge of $\mathcal D$ is a connecting edge of a unique band.

\begin{lemma}\label{nocross} %{\bf (nocross)}
A band cannot cross itself. In other words, if $e_1$ and $e_2$ are opposite edges labeled $a$ in a cell, and $d_1$ and $d_2$ the other $a$-edges of that cell, then the band for $e_1$ and $e_2$ does not contain $d_1$ or $d_2$ as connecting edges. 
\end{lemma}

\begin{proof}
Say there is a collection of cells $C_i$, $i\in\{1,\ldots, n\}$ such that $e_i$ and $e_{i+1}$ are opposite edges in $C_i$ and $e_{n+1}$ is an $a$-edge of $C_1$ (not equal to $e_1)$ or $e_2$. 

The edge path on the boundary of the band from  $q$, the initial point of $e_2$ to $y$, the initial point of $e_{n+1}$ (marked $\alpha$ in Figure 2) has 0-exponent sum on the $x$-letters. But there is also an edge path between $q$ and $y$ in $C_1$ with labeling $(a=e_2, x^{-k}, a=e_{n+1})$  and $x$-exponent sum $-k$.

The group elements of $L$ represented by the labelings of these two edge paths are the same, but that is impossible as the homomorphism from $L$ to $\mathbb Z$ that quotients out the normal closure of $a$ maps one of these these elements to $0$ and the other to $-k$. 
\end{proof}

\vspace {.5in}
\vbox to 2in{\vspace {-2in} \hspace {-.3in}
\hspace{-.5 in}
\includegraphics[scale=1]{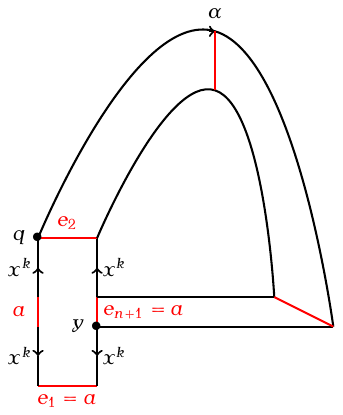}
\vss }

\vspace{.4in}

\centerline{Figure 2}

\begin{remark}\label{smaller} %{\bf (smaller)} 
It is clear that a band that begins on the boundary of a diagram ends on the boundary of the diagram and if a band does not have a boundary edge,  then it forms an annulus. Notice that if $e_1$ and $e_2$ are $a$-edges that are opposite one another in a cell $C$, then the two sides of $C$ (between $e_1$ and $e_2$) have the same labeling $(x^{-k}ax^k)$. Hence if a band forms an annulus in a diagram, then the cells of the band can be removed to form a smaller diagram for the same boundary curve. 
\end{remark}

Suppose $Z$ is a locally finite CW complex. If $Q$ is a subcomplex of $Z$ we define $st(Q)$ to be the subcomplex of $Z$ with vertex set equal to all vertices of $Q$ union all vertices of $Z$ that are connected to a vertex of $Q$ by an edge. A cell of $Z$ belongs to $st(Q)$ if all of its vertices belongs to $st(Q)$. Inductively, $st^k(Q)=st(st^{k-1}(Q))$ for all $k\geq 2$. Any compact subset of $Z$ belongs to a finite subcomplex of $Z$ and if $A$ is a finite subcomplex of $Z$ then $st(A)$ is a finite subcomplex of $Z$. 

\begin{theorem}\label{main} %{\bf(main)}
The Lamplighter group is not semistable at $\infty$.
\end{theorem}
\begin{proof} 
Any finite set of relators for $L$ is a consequence of the set of relators $\mathcal R_{n-1}=\{a^2, (ax^{-1}ax)^2, \ldots, (ax^{n-1}ax^{-(n-1)})^2\}$ for some integer $n>0$. Suppose $\Gamma=\Gamma(T, \{x,a\}, \mathcal R_{n-1})$ is semistable at $\infty$. 
Let $\ast$ be the identity vertex of $\Gamma$. Recall the vertices of $\Gamma$ are the elements of $L$. Let $K$ be an integer such that the finite subgroup 
$$\langle a, x^{-1}ax,\ldots, x^{-(n-1)}ax^{n-1}\rangle\subset st^K(\ast).$$
Let $r$ be the edge path ray at $\ast$ each of whose edges is labeled by $x$. Let $C$ be a finite subcomplex of $\Gamma$ and $D$ be a finite subcomplex so that for any finite complex $A$ and any loop $\alpha$ on $r$ with image in $\Gamma-D$, $\alpha$ is homotopic $rel\{r\}$ to a loop in $\Gamma-A$ by a homotopy in $\Gamma -C$. Let $v$ be a vertex of $r$ such that the tail of $r$ at $v$ avoids $D$ and $v$ is not in $st^{4n+4}(D)$.  Consider the loop $\alpha$ at $v$ with edge path label $(a, x^{-n}, a, x^n, a, x^{-n}, a, x^n)$, then since $|\alpha|=4n+4$, $\alpha$ avoids $D$. 

Let $A=st^N(\ast)$, where $N$ is large enough so that
 $$st^{4n+4}(D)\cup st^{K+2n+1}(v)\subset A$$ 
We show that $\alpha$ is not homotopic $rel\{r\}$ to a loop in $\Gamma-A$ (let alone by a homotopy in $\Gamma-C$). This will be the desired contradiction to the semistabilty assumption on $\Gamma$. 

Suppose otherwise and $\alpha$ is homotopic $rel\{r\}$ to the loop $\beta$ in $\Gamma-A$. This means that $\alpha$ is homotopic to the path with labeling $(x^k, \beta, x^{-k})$ (for some large integer $k>0$).  By Lemma \ref{L1}, there is a van Kampen diagram $\mathcal D$ for the edge path loop $\langle \alpha, x^k, \beta^{-1},x^{-k}\rangle$ such that each 2-cell of $\mathcal D$ is labeled by an element of $\mathcal R_{n-1}^{\pm 1}$. 
There are four edges labeled $a$ on $\alpha$. Say the vertices of the first are $v,w_1$, the second are $v_2,w_2$ etc (see Figure 3). Consider the $a$-band $B_1$ that starts on the first $a$-edge of $\alpha$. We identify bands in $\mathcal D$ with the corresponding bands in $\Gamma$. The sides of $B_1$ are labeled by elements of $\{x^{-1}ax,\ldots, x^{-(n-1)}ax^{n-1}\}$, and so each connecting $a$-edge of the band is in $st^K(v)$.  
This implies that each vertex of  $B_1$ is in $st^{K+n}(v)\subset A$. Hence $B_1$ cannot end on $\beta$.
Instead $B_1$ must end at one of the $a$-edges of $\alpha$. 
The band $B_1$ cannot end at the second $a$-edge of $\alpha$ since otherwise the edge path from $w_1$ to $v_2$ along $\alpha$ has $x$-exponent sum equal to $-n$ and the path from $w_1$ to $v_2$ along $B_1$ has $x$-exponent sum equal to $0$ (see Figure 3). Similarly $B_1$ cannot end at the fourth $a$ edge of $\alpha$. Instead, it ends at the third $a$-edge of $\alpha$. Next consider the band $B_2$ at the second $a$-edge of $\alpha$. Each connecting $a$-edge of $B_2$ is in $st^K(v_2)$ and so each vertex of $B_2$ is in $st^{K+n}(v_2)$. This implies $B_2$ is in $st^{K+2n+1}(v)$. Hence $B_2$ cannot end on $\beta$. Then $B_2$ ends on the fourth $a$-edge of $\alpha$. (See Figure 3) The band $B_1$ crosses $B_2$ at an $a$-edge with initial vertex $y$. The edge path along $B_1$ that starts at $w_1$ and ends at $y$ has $x$-exponent sum equal to $-m$ for $m<n$. The edge path beginning at $w_1$ following $\alpha$ by: $x^{-n}$ (to $v_2$) and then following $B_2$ from $v_2$ to $y$ has $x$-exponent sum equal to $-n$. This defines two words representing the same element in $L$ with different $x$-exponent sums, which is impossible.
\end{proof}

\vspace {.5in}
\vbox to 2in{\vspace {-2in} \hspace {-.3in}
\hspace{-1 in}
\includegraphics[scale=1]{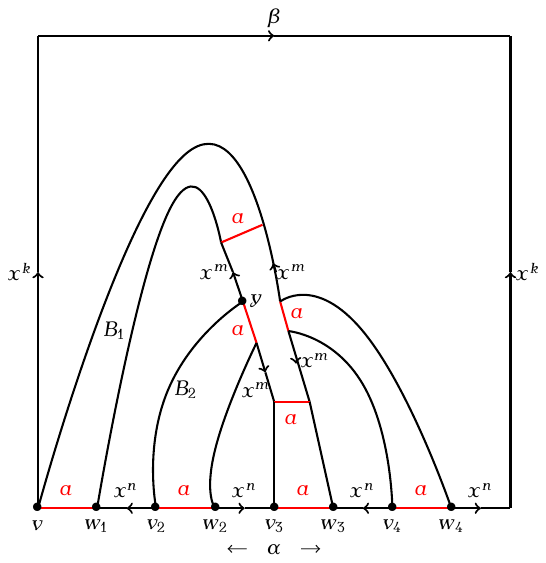}
\vss }

\vspace{1.4in}

\centerline{Figure 3}

\medskip

\section{Potential Generalizations}\label{O}

Denis Osin suggests that the ideas behind our proof may extended to a class of groups which contains all finitely generated groups $G$ that fit into a exact sequence:
$$1\to K\to G\to F\to 1$$ 
where $K$ is a locally finite infinite group and $F$ is an infinite finitely generated free group as well as finitely generated not finitely presented limits of virtually free groups. 

For this last class of groups there is a sequence of epimorphisms: 
$$F_0{\buildrel q_1\over  \twoheadrightarrow} F_1 {\buildrel q_2 \over \twoheadrightarrow}\cdots $$
where each $F_i$ is a finitely presented infinite virtually free group. If $K_i$ is the kernel of $q_i$ and $K^i$ is the preimage of $K_i$ in $F_0$ then $K^1<K^2<\ldots$ is an ascending sequence of normal subgroups in $F_0$. Let $K=\cup_{i=1}^\infty K^i$ then the limit of the sequence of epimorphisms is  the group $G=F_0/K$. 
%Important facts here: Each $F_i$ has more than 1-end and is word hyperbolic, and the group $G$ has the ``bottleneck" property. %There are interesting elements of infinite order in $G$...

\section{The Extended Lamplighter Group}\label{E}

Recall that the Lamplighter group $L$ has presentation 
$$\langle x, a: a^2=1, [a, x^{-k}ax^k]=1 \hbox{ for all integers } k\rangle .$$ 
There is a monomorphism $L\to L$ taking $x\to x$ and $a\to x^{-1}axa$. The resulting ascending HNN extension $E$ is called the Extended Lamplighter group and has finite presentation:
$$E=\langle x,a,t: a^2=1, [x,t]=1, t^{-1}at=x^{-1}axa\rangle.$$
There is a short exact sequence:
$$1\to N\to E\to \mathbb Z\times \mathbb Z\to 1$$ 
where $N$, the normal closure of $a$ in $E$, is an infinite direct sum of copies of $\mathbb Z_2$). If $u,v\in E$, then $uvu^{-1}v^{-1}\in N$ and so $(uvu^{-1}v^{-1})^2=1$. In particular, $E$ contains no free group of rank 2.
A result of R. Bieri, W. Neumann and R. Strebel [Theorem D, \cite{BNS87}] implies that $E$ contains an infinite finitely generated normal subgroup $K$ such that $E/K$ is isomorphic to $\mathbb Z$. Combining this result with the main theorem of \cite{M1} implies that $E$ is semistable at infinity. But more can be said about $E$. 
In his PhD thesis, N. Silkin 
\cite{Silk1} 
proved that $E$ is simply connected at $\infty$, a property stronger than semistability. 
%See Lamplighter.tex to see more about the group 
%$$E_2=\langle x,y,t: y^2=1, [x,t]=1, [y,x^{-1}yx]=1, x^{-1}yt=x^{-2}yx^2y\rangle$$
%is simply connected at $\infty$.

\bibliographystyle{amsalpha}

\bibliography{paper}{}

\end{document}